\documentclass[12pt,a4paper,reqno,oneside]{amsart}

\newtheorem{thm}{Theorem}[section]
\newtheorem{prop}[thm]{Proposition}

\newtheorem{cor}[thm]{Corollary}

\theoremstyle{remark}
\newtheorem{rem}[thm]{Remark}

\theoremstyle{definition}
\newtheorem{defn}[thm]{Definition}

\newcommand{\cl}{c\`adl\`ag}
\renewcommand{\phi}{\varphi} 
\newcommand{\E}{\mathrm{E}} 
\newcommand{\<}{\langle} 
\renewcommand{\>}{\rangle}
\newcommand{\bbc}{\mathbb{C}}
\newcommand{\bbr}{\mathbb{R}}
\newcommand{\bbq}{\mathbb{Q}}

\newcommand{\bbn}{\mathbb{N}}

\newcommand{\abs}[1]{\left| #1 \right|}

\usepackage[final]{changes}
\colorlet{Changes@Color}{red}

\begin{document}

\title{Time change equations for L\'evy type processes}
\date{\today}

\author[Kr\"uhner]{Paul Kr\"uhner}
\address[Paul Kr\"uhner]{\\
Vienna University of Technology\\
FAM -- Financial and Actuarial Mathematics\\
Wiedner Hauptstra\ss e 8-10\\
AT--1040 Vienna, Austria}
\email[]{paulkrue\@@fam.tuwien.ac.at}
\urladdr{https://fam.tuwien.ac.at/~paulkrue/}
\author[Schnurr]{Alexander Schnurr}
\address[Alexander Schnurr]{\\
Technische Universit\"at Dortmund\\
Fakult\"at f\"ur Mathematik\\
Lehrstuhl LSIV\\
Vogelpothsweg 87\\
DE--44227 Dortmund, Germany}
\email[]{alexander.schnurr\@@math.tu-dortmund.de}
\urladdr{http://www.mathematik.tu-dortmund.de/de/personen/person/Alexander+Schnurr.html}




\begin{abstract}
In this paper we analyse time change equations (TCEs) for L\'evy-type processes in detail. 
To this end we establish a connection between TCEs and classical one-dimensional initial value problems (IVPs) which are easier to handle. 
Properties of the IVPs are linked with properties of the TCEs. 
We show in a general setting existence and uniqueness of solutions of the TCEs.
Our main result is based on the general path properties for L\'evy-type processes found in Schnurr (2013). 
Applications include an existence result for processes which correspond to a certain class of given symbols. 
\end{abstract}
\maketitle


\section{Introduction}
The study of multiplicative perturbation has started with early papers like Dorroh \cite{dorroh.66} and the more general versions by Gustafson and Lumer \cite{gustafson.lumer.72}, see also Jacob \cite{jacob.93}. Dorroh has focused on the very relevant contraction semigroups on continuous function spaces perturbed with a multiplier where the multiplying function is continuous, bounded and strictly positive. Pre-dating the paper of Dorroh, Volkonskii \cite{volkonskii.58} has found a path transformation between Brownian motion and continuous sample path Markov processes on the real line. The similar transformation used by Lamperti \cite{lamperti.72} relates positive self-similar processes and L\'evy processes. The generalization of this path transformation connects the analytic multiplicative perturbation theory with pathwise transformation of the corresponding stochastic processes. Ethier and Kurtz \cite[Section 6]{ethier.kurtz.86} have investigated in their 
book the connection of the stochastic transform to the analytic multiplicative perturbation. In the present work we specialize it to L\'evy-type processes which are characterized by their symbol $q$, cf.\ Jacob and Schilling \cite{jacob.schilling.01} which can be seen as an encoding of the state based characteristics $(b(x),c(x),F(x,\cdot))_{x¸\in E}$ via $q(x,u) = i\< u,b(x)\> - \frac{1}{2}\<c(x)u,u\> + \int_{\mathbb R^d} \left(e^{i\<u,y\>}-1-i\<u,\chi(y)\>\right)F(x,dy)$, cf.\ Proposition \ref{p:semimmartingale} below. B\"ottcher, Schilling and Wang \cite{boettcher.al.13} summarise that if $(q(x,u))_{x\in E,u\in\mathbb R^d}$ is a symbol that belongs to a Markov process and $\beta:E\rightarrow \mathbb R$ is continuous, bounded and strictly positive, then $(\beta(x)q(x,u))_{x\in E,u\in\mathbb R^d}$ is a valid symbol. This is essentially the translation of Dorroh's result to the stochastic setup. In the book of Ethier and Kurtz \cite[Section 6]{ethier.kurtz.86} this approach is more general in the sense that they do allow for the multiplying function $\beta$ to be only measurable and non-negative, however, the technical condition in their Theorem 1.1 cannot be verified easily from the symbol and the multiplying function. Engelbert and Schmidt \cite{engelbert.schmidt.85} optimise the original approach of Volkonskii \cite{volkonskii.58} and find exact conditions under which a multiplicative perturbed Brownian motion gives rise to a strong Markov process. Their approach is based on detailed knowledge of the Brownian local time or in some sense on path properties of the Wiener process. Path properties of L\'evy-type processes have been studied in the paper of Schnurr \cite{schnurr.13} and are utilised herein to improve the result in B\"ottcher et al. \cite{boettcher.al.13} in two ways. First, we do allow that the multiplying function hits zeros and, second, we allow for measurable instead of continuous multiplying functions. For a continuous multiplier we essentially need that the function does not grow too quickly near its zero in order to ensure that these points become absorbing states. This, allows to adapt the arguments in B\"ottcher et al. \cite{boettcher.al.13} or Ethier and Kurtz \cite{ethier.kurtz.86}.

Applications of random time changes include the seminal work of Volkonskii \cite{volkonskii.58}, its perfection in Engelbert and Schmidt \cite{engelbert.schmidt.85}, the work of Lamperti \cite{lamperti.72} to identify self-similar Markov process, see also D\"oring \cite{doering.15} and application to affine processes by Kallsen \cite{kallsen.04} and Gabrielli and Teichmann \cite{gabrielli.teichmann.15}.

The question for which symbols $q:E\times \bbr^d\to\bbc$ (or classes of symbols), there exists a corresponding stochastic process is a vital part of ongoing research. Compare in this context: Jacob and Schilling \cite{jacob.schilling.01} and B\"ottcher et al. \cite{boettcher.al.13}. Techniques in order to establish existence results include approaches via Dirichlet forms, the Hille-Yosida-Ray theorem and solutions to the martingale problem. Here, we contribute to this part of the theory using an approach via time changes. If one can proof by any of the above techniques that for the symbol $q$, there exists a corresponding process $X$ then we get for the whole class of symbols which can be written as $\beta(x)q(x,u)$, with $\beta$ as described below, that corresponding processes do exist.  

The time change equations which are used in the present article are a certain kind of random time changes. They have to be distinguished from other random time changes like Bochner's subordination. In this latter concept an \added{independent} increasing process $L(u)$  serves as a new time scale of the process $X$, that is, $(X(L(u)),u\geq 0)$ is being considered. There is a wast literature on this subject \replaced{, cf.\ \cite{schilling.et.al.12},}{(cf. XXX)} which has caught renewed interest recently, cf. Deng and Schilling \cite{dengschilling}. Let us mention that it has become common in the context of mathematical statistics to call the subordination just `time-change' (cf. e.g. Belomestny \cite{belomestny}). 

This paper is organized as follows. First we clarify some of our notations. In the second section we recall some definitions and known results. The third section contains our main results along with the proofs.

\subsection{Mathematical preliminaries}
$\mathbb R$ resp.\ $\mathbb C$ denote the real respectively complex numbers. We denote the trace of a matrix $C\in\mathbb R^{d\times d}$ by $\mathrm{Tr}(C)$.
The set of positive semidefinite $d\times d$-matrices is denoted by $S^d$. A truncation function is a compactly supported, bounded and continuous function $\chi:\mathbb R^d\rightarrow\mathbb R^d$ which equals the identity function on a neighbourhood of $0$. We will use the continuous function $\chi:\mathbb R^d\rightarrow \mathbb R^d,x\mapsto x1_{\{\vert x\leq 1\}}+\frac{x}{\vert x\vert}1_{\{\vert x\vert >1 \}}$ as truncation function. $\vert x\vert^2:=\sum_{j=1}^d\vert x_j\vert^2$ denotes the Euclidean norm on $\mathbb R^d$. For any $x\in\mathbb R^d$, $r>0$ we denote the open ball with radius $r$ centred at $x$ by $B(x,r) := \{y\in\mathbb R^d: \vert x-y\vert< r\}$. 
Further unexplained notations are used as in the book of Ethier and Kurtz \cite{ethier.kurtz.86}.

\section{Markov processes and symbols}
In this section we recall the definition of the symbol and the basic connection between symbols and their processes. A symbol describes a (hopefully unique) generator of a Markov process restricted to a smaller domain, namely to the set of functions with Fourier representation. 
A Markov process belonging to a symbol is a semimartingale. The symbol can be seen as a simple encoding of a state based version of the semimartingale characteristics which mimics the L\'evy-Khintchine-formula for L\'evy processes. Various properties of Markov processes belonging to a given symbol have been found in the literature, see e.g.\ the monograph of B\"ottcher et al. \cite{boettcher.al.13} or the survey paper \cite{jacob.schilling.01} from Jacob and Schilling.

We first start with the basic definition of the symbol of a process.
\begin{defn}\label{d:symbol}
 Let $E\subseteq \mathbb R^d$ be a measurable set. A {\em Markov-triplet on $E$} is a triplet $(b,c,F)$ such that $b:E\rightarrow \mathbb R^d$, $c:E\rightarrow S^d$ are measurable and $F:E\times \mathcal B(\mathbb R^d)\rightarrow \mathbb R_+$ is a transition kernel such that $\int_{\mathbb R^d} \vert\chi(y)\vert^2 F(x,dy)$ is finite for any $x\in\mathbb R^d$. A function $q:E\times\mathbb R^d\rightarrow\mathbb C$ is a {\em symbol} if there is a Markov-triplet $(b,c,F)$ such that
  $$ q(x,u) = i\< u,b(x)\> - \frac{1}{2}\<c(x)u,u\> + \int_{\mathbb R^d} \left(e^{i\<u,y\>}-1-i\<u,\chi(y)\>\right)F(x,dy)$$
  for any $x\in E$, $u\in\mathbb R^d$. We say that $(b,c,F)$ is the triplet \emph{associated with} $q$ which is unique by \cite[Lemma II.2.42]{js.87}. \deleted{A \emph{L\'evy type process} is ... bitte hier einfuegen}
 Let $X$ be a \cl\ stochastic process with values in $E$. The function $q$ is the {\em symbol} of $X$ if $\E\int_0^t \vert q(X(s),u)\vert ds <\infty$ for any $t\geq0$, $u\in\mathbb R^d$ and
  \begin{align*}
     M_u(t) &:= e^{i\<u,X(t)\>} - \int_0^t e^{i\<u,X(s)\>}q(X(s),u) ds,\quad t\geq0
  \end{align*}
 is a martingale for any $u\in\mathbb R^d$.
 
\added{A \emph{L\'evy type process} is a strong Markov process $X$ with state space $E\subseteq\mathbb R^d$ on $(\Omega,\mathcal A,(\mathcal F_t)_{t\geq 0},(P_x)_{x\in E})$ such that $X$ has the same symbol $q:E\times\mathbb R^d\rightarrow\mathbb C$ under $P_x$ for any $x\in E$.}
\end{defn}

\begin{rem}
 If $X$ is a stochastic process with values in a measurable set $E\subseteq \mathbb R^d$ and $q_1,q_2:E\times\mathbb R^d\rightarrow\mathbb C$ are both symbol of $X$, then
  $$ \int_0^t e^{i\<u,X(s)\>}(q_1(X(s),u)-q_2(X(s),u)) ds,\quad t\geq0 $$
 is a local martingale and, hence, it is constant zero. Thus, the processes $q_1(X,u)$, $q_2(X,u)$ are indistinguishable for any $u\in\mathbb R^d$.
 
 In other words, the symbol of a process $X$ might have several versions $q_1,q_2$ but they coincide in the sense that the processes $q_1(X,u)$, $q_2(X,u)$ are indistinguishable.
\end{rem}

In this paper we are mainly interested in L\'evy type processes. The following proposition establishes the connection to the definition of a symbol of a process used in \cite{jacob.schilling.01}.
\begin{prop}
  Let $X$ be a L\'evy-type process with \added{continuous} symbol $q$. Then, we have
   $$ \lim_{t\searrow0} \frac{\E_x(e^{i\<u,X(t)-x\>})-1}{t} = q(x,u) $$
  for any $x\in E$, $u\in\mathbb R^d$.
\end{prop}
\begin{proof}
Let $x\in E$, $u\in\mathbb R^d$. We have
 \begin{align*}
  \E_x(e^{i\<u,X(t)-x\>})-1 &= e^{-i\<u,x\>} \E_x\left( \int_0^t e^{i\<u,X(s)\>}q(X(s),u)  ds\right) \\
    &= \int_0^t \E_x(e^{i\<u,X(s)-x\>}q(X(s),u)) ds
 \end{align*}
where we used the $P_x$-martingale property of $M_u$ given in Definition \ref{d:symbol}. Right-continuity of the integrand and the fundamental theorem of calculus yield the claim.
\end{proof}

Another view is that the process $X$ is a solution to a certain martingale problem.
\begin{rem}
 Let $X$ be a strong Markov process with symbol $q$ and define $A := \{ (f_u,q(\cdot,u)f_u) : u\in\mathbb R^d\}$ where $f_u:E\rightarrow\mathbb C,x\mapsto e^{i\<u,x\>}$. Then, the process $X$ on $(\Omega,\mathcal A,\mathcal (F_t)_{t\geq 0},P_x)$ is a solution to the martingale problem $(\mathcal A,\delta_x)$ in the sense of \cite[p.~173]{ethier.kurtz.86} for any $x\in E$.
\end{rem}

Finally, we like to recall that a strong Markov process $X$ with a symbol $q$ is a semimartingale and a version of its characteristics is described by the Markov triplet associated with $q$.
\begin{prop}\label{p:semimmartingale}
  Let $X$ be a L\'evy-type process with symbol $q$ and $(b,c,F)$ be the triplet associated with $q$. Then, $X$ is a semimartingale and 
  \begin{align*}
    B(t) & := \int_0^t b(X(s)) ds, \\
    C(t) &:= \int_0^t c(X(s)) ds, \\
    \nu(dx,dt) &:= F(X(t),dx)dt
  \end{align*}
  is a version of the characteristics of $X$ relative to the truncation function $\chi$ in the sense of \cite[Definition II.2.6]{js.87}.
\end{prop}
\begin{proof}
  This is a direct corollary to \cite[Theorem II.2.42]{js.87}.
\end{proof}

\section{Random time changes}
In this section we postulate our main result, see Theorem \ref{t:main statement} below. It is widely known that multiplicative perturbation for generators of Feller semigroups are connected to random time changes, cf.\ the book of B\"ottcher et al. \cite[Chapter 4.1]{boettcher.al.13}. Multiplicative perturbation theorems for $c_0$-semigroups have been studied for instance by Gustafson and Lumer \cite{gustafson.lumer.72} and Jacob \cite{jacob.93}. The more stochastic perspective on this concept has been studied in the book of Ethier and Kurtz \cite[Chapter 6]{ethier.kurtz.86} and, of course, in Engelbert and Schmidt \cite{engelbert.schmidt.85} where the latter has a focus on time change equation for the Brownian motion. See also the book of Karatzas and Shreve \cite[Chapter 5.5]{karatzas.shreve.98}. 

Throughout this section \replaced{let $X$ be a L\'evy-type process with symbol $q:E\times\mathbb R^d\rightarrow\mathbb C$ on $(\Omega,\mathcal A,(\mathcal F_t)_{t\geq0},(P_x)_{x\in E})$ where $E\subseteq\mathbb R^d$ is its state space and $P_x$ denotes the probability measure with $P_x(X(0)=x)=1$.}{let $X$ be a strong Markov process {\bf [sollte das ab hier nicht alles levytype sein?]} with c\`adl\`ag paths on a filtered probability space $(\Omega,\mathcal A,(\mathcal F_t)_{t\geq0},(P_x)_{x\in E})$ where $E\subseteq\mathbb R^d$ is its state space and $P_x$ denotes the probability measure with $P_x(X(0)=x)=1$. Moreover, we assume that $X$ has symbol $q:E\times\mathbb R^d\rightarrow\mathbb C$.} In Schnurr \cite{schnurr.13} it is demanded that the symbol (respectively the differential characteristics are finely continuous (see Blumenthal and Getoor \cite{blumenthalget} Section II.4 and Fuglede \cite{fuglede}). In that article, however, this property is only used to derive the existence of a symbol. This is established here in a different way. Therefore, we can use the maximal inequality (12) of Schnurr \cite{schnurr.13} in our context, too. 

For a probability measure $\mu$ on $E$ we will denote the measure $P_\mu(A):=\int_EP_x(A)d\mu$, note that the family of measures $(P_x)_{x\in E}$ is measurable because $X$ is a strong Markov process. 

\begin{defn} 
 Let \deleted{$X$ be an $E$-valued stochastic process and} $g:E\rightarrow\mathbb R_+$. A stochastic process $Z$ is a {\em solution to the time change equation (TCE) $(X,g)$} if
\begin{align} \label{tce} \tag{TCE}
 Z(t) = X\left(\int_0^t g(Z(s)) ds\right),\quad t\geq 0
\end{align}
    \replaced{$P_x$-a.s.\ for any $x\in E$.}{$P$-a.s. {$P^\mu$oder?}} We say that {\em uniqueness} holds for the TCE $(X,g)$ if for any two solutions $Z_1,Z_2$ are indistinguishable. We say that {\em existence} holds if the TCE $(X,g)$ has a solution.
    
    Let $f:\mathbb R_+\rightarrow\mathbb R_+$ and $x\in\mathbb R_+$ a {\em solution to the initial value problem (IVP) $(f,x)$} is a function $y:\mathbb R_+\rightarrow\mathbb R_+$ such that 
\begin{align} \tag{IVP}
y(t) = x+\int_0^tf(y(s)) ds.
\end{align}
\end{defn}
\begin{rem}
Time change equations are extensively studied in the book of Ethier and Kurtz, see \cite[Section 6.1,6.2]{ethier.kurtz.86}. 
Their criteria are quite abstract and they are not easy to apply in practice. In contrast to this our main result (in particular in the form of Corollary  \ref{main cor} can be applied easily. 
\end{rem}

Now we recall the following quantities from Schnurr \cite{schnurr.13} which generalises results from Schilling \cite{schilling98}. In these articles only processes on $\bbr^d$ are considered. It does not pose a problem to consider processes on general state space $E$, as we do it here. Every process defined on $E$ can be prolonged to $\bbr^d$ by setting $X_t^x=x$ for $t\geq 0$ and $x\in \bbr^d \backslash E$. Measurability and (local) boundedness of the symbol, as well as normality are inherited by the process defined on $\bbr^d$.  For $x\in\bbr^d$ and $R>0$ we set:
\begin{align}
H(x,R)&:= \sup_{\abs{y-x}\leq 2R} \sup_{\abs{\varepsilon}\leq 1} \abs{q\left(y,\frac{\varepsilon}{R}\right)} \\
H(R)&:= \sup_{y\in\bbr^d} \sup_{\abs{\varepsilon}\leq 1} \abs{q\left(y,\frac{\varepsilon}{R}\right)}
\end{align}
and the uniform index
\[
  \beta_\infty:=\inf \left\{\lambda > 0 : \limsup_{R\to 0} R^\lambda H(R) =0 \right \}\in[0,2].
\]

We now postulate our main result.
\begin{thm}\label{t:main statement}
  Let $g:E\rightarrow \mathbb R_+$ be bounded, measurable and regular at zero, i.e.\ for any $x\in E$ such that $g(x)>0$ there are $\delta>0$, $\epsilon>0$ such that $g(y)>\epsilon$ for any $y\in B(x,\delta)\cap E$. Let $\lambda>\beta_\infty$ and assume that there is a constant $C_g>0$ such that for any $x\in E$ with $g(x)=0$ there is $\delta>0$ such that for any $y\in B(x,\delta)\cap E$ we have
   $$ g(y) \leq C_g\vert y-x\vert^\lambda. $$
   
Then, the TCE $(X,g)$ has a unique solution. Moreover, if $Z$ is the solution to the TCE $(X,g)$, then it is a process with symbol $q_Z(x,u):=g(x)q(x,u)$ for any $x\in E$, $u\in\mathbb R^d$. If the symbol $q$ determines the law of $X$ and $q_Z$ is continuous, then $Z$ is a strong Markov process on $(\Omega,\mathcal A,(\mathcal F_t^Z)_{t\geq 0},(P_x)_{x\in E})$ in the sense of \cite[p.~158]{ethier.kurtz.86}.
\end{thm}
The proof for this theorem will follow from Corollary \ref{k:hinreichende Bedingungen} and Propositions \ref{p:(A2) holds}, \ref{p:regular solutions} below and it is broken into several steps. We start by analysing a related IVP. Then, we continue to show how the IVP is connected to the TCE and how solutions of the TCE give rise to a strong Markov process. Finally, we slightly extend results found in \cite[Section 3]{schnurr.13} on path properties for the process $X$. First we draw a corollary from our main statement.

\begin{cor} \label{main cor}
 Let $g:E\rightarrow \mathbb R_+$ be bounded and continuous. Assume that there is $C_g>0$ and $\epsilon>0$ such that $\vert g(y)\vert \leq C_g\vert x-y\vert^{2+\epsilon}$ for any $x,y\in E$ with $g(x)=0$.
 
 Then, the TCE $(X,g)$ has a unique solution. Moreover, if $Z$ is the solution to the TCE $(X,g)$, then it is a process with symbol $q_Z(x,u)=g(x)q(x,u)$ for any $x\in E$, $u\in\mathbb R^d$. If the symbol $q$ determines the law of $X$ and $q_Z$ is continuous, then $Z$ is a strong Markov process on $(\Omega,\mathcal A,(\mathcal F_t^Z)_{t\geq 0},(P_x)_{x\in E})$ in the sense of \cite[p.~158]{ethier.kurtz.86}.
\end{cor}
\begin{proof}
  Any continuous function $g$ is regular at zero. Moreover, $2+\epsilon>2\geq\beta_\infty$ and, hence, the requirements of Theorem \ref{t:main statement} are met.  \deleted{\bf 3.3 liefert nur MP, wieso gilt hier Feller. Verwenden wir hier nicht van Casteren?}
\end{proof}

We start to show that there are maximal and minimal solution systems for a certain class of IVPs. Later, these IVPs will show up naturally in the study of TCEs.

\begin{prop}\label{p:Existenz IVP}
 Let $Y:\Omega\times\mathbb R_+\rightarrow \mathbb R_+$ be bounded and measurable and assume that $Y$ is right regular at zero, i.e.\ for any $(\omega,t)\in\Omega\times\mathbb R_+$ with $Y(\omega,t)>0$ there are $\epsilon,\delta>0$ such that $Y(\omega,s)>\epsilon$ for any $s\in (t,t+\delta)$. Then, there are measurable functions $\alpha_1,\alpha_2:\Omega\times\mathbb R_+\rightarrow\mathbb R_+$ such that
 \begin{itemize}
  \item $\alpha_i(\omega,t) = \int_0^t Y(\omega,\alpha_i(\omega,s)) ds$ for any $t\geq 0$, $\omega\in\Omega$, $i=1,2$ and
  \item For any fixed $\omega\in\Omega$ we have $\alpha_1(\omega,\cdot)$ (resp.\ $\alpha_2(\omega,\cdot)$) is the minimal (resp.\ maximal) solution of the IVP $(Y(\omega,\cdot),0)$, i.e.\ if $z(t) = \int_0^t Y(\omega,z(s))ds$ for any $t\geq 0$, then $\alpha_1(\omega,t)\leq z(t)\leq \alpha_2(\omega,t)$ for any $t\geq 0$.
 \end{itemize}
 
 Define the $[0,\infty]$-valued functions
 \begin{align*}
   \tau &:= \inf\{t\geq 0: Y(t)=0\}, \\
   \eta &:= \inf\left\{t\geq 0: \int_0^t\frac{1}{Y(s)}ds=\infty\right\}.
 \end{align*}
 Then $\eta(\omega)\leq \tau(\omega)$ for any $\omega\in\Omega$ if and only if $\alpha_1=\alpha_2$.
 
 Moreover, if for some $\omega\in\{\tau<\infty\}$ we have $\int_{\tau(\omega)}^{\tau(\omega)+\epsilon}\frac{1}{Y(\omega,s)}ds = \infty$ for any $\epsilon >0$, then $\eta(\omega) \leq \tau(\omega)$.
\end{prop}
\begin{proof}
 We start with the construction of the functions $\alpha_1$, $\alpha_2$.

Since $Y$ is regular at zero we have  $Y(\omega,\tau(\omega))=0$ for $\omega\in \{\tau<\infty\}$. Define 
 $$I(\omega,t) := \int_0^t \frac{1}{Y(\omega,s)} ds $$
 for any $\omega\in\Omega$, $t<\eta(\omega)$. Then $I(\omega,\cdot)$ is strictly increasing on $[0,\eta)$ and, hence, has an inverse $g$ defined on $[0,\gamma)$ where $\gamma:=\sup\{I(t):t<\eta\}$. Observe that $I$ is a strictly increasing continuous function with absolutely continuous derivative bounded from below by $1/\Vert Y\Vert_\infty$ where $\Vert Y\Vert_\infty:=\sup_{\omega,t}\vert Y(\omega,t)\vert$. Thus, $g$ is Lipschitz-continuous with absolutely continuous derivative bounded by $\Vert Y\Vert_\infty$. Let $\omega\in\Omega$, $t\in [0,\gamma(\omega))$ and define $u:=g(\omega,t)$. Denoting a version of the absolutely continuous derivative of $g(\omega,\cdot)$ by $g'(\omega,\cdot)$ we get
  $$ t=I(\omega,u) = \int_0^u\frac{1}{Y(\omega,x)}dx = \int_0^{t} \frac{g'(\omega,s)}{Y(\omega,g(\omega,s))}ds $$
 for any $t\geq 0$ where we used the substitution formula (with $x=g(\omega,s)$) for the third equation. Thus, $g'(\omega,s) = Y(\omega,g(\omega,s))$ for Lebesgue almost every $s\in[0,\gamma(\omega))$. For the remainder we use the nicer version $Y(\omega,g(\omega,\cdot))$ for the absolutely continuous derivative of $g(\omega,\cdot)$. We have
  $$ g(\omega,t) = \int_0^t Y(\omega,g(\omega,s)) ds $$
for any $\omega\in\Omega$, $t\in[0,\gamma(\omega))$.

 Define the $[0,\infty]$-valued functions
   \begin{align*}
      \alpha_1(\omega,t) &:= \begin{cases} g(\omega,t)\wedge\tau(\omega) & t< \gamma(\omega), \\ \tau(\omega) & \text{otherwise},\end{cases} \\
      \alpha_2(\omega,t) &:= \begin{cases} g(\omega,t) & t< \gamma(\omega), \\ \eta(\omega) & \text{otherwise}, \end{cases} \\
   \end{align*}
  for any $\omega\in\Omega$, $t\geq 0$. 
  
Now we like to see that the functions $\alpha_1$, $\alpha_2$ do not attain the value infinity.

Clearly, $\alpha_1,\alpha_2$ are finite-valued on $[0,\gamma)$. Thus, they are finite valued on $\{\gamma=\infty\}$. Let $\omega\in\{\gamma<\infty\}$. Then, $\eta(\omega)<\infty$ and we have $\int_0^{\eta(\omega)}\frac{1}{Y(\omega,s)}ds<\infty$. 

Assume by contradiction that $Y(\omega,\eta(\omega))>0$. Then there is $\epsilon>0$ such that $\int_0^{\eta(\omega)+\epsilon}\frac{1}{Y(\omega,s)}ds<\infty$. Therefore, we obtain a contradiction. Thus, $\tau(\omega)\leq \eta(\omega)<\infty$ for any $\omega\in\{\gamma<\infty\}$.

Consequently, $\alpha_1,\alpha_2$ are finite-valued.
  
Measureability of $\alpha_1,\alpha_2$ as well as the integral representation follow from their constructions.

Finally, we like to show that any other solution is bounded from below by $\alpha_1$ and bounded from above by $\alpha_2$. 
  
  To this end, let $\omega\in\Omega$ and $z:[0,\infty)\rightarrow[0,\infty)$ such that $z(t) = \int_0^tY(\omega,z(s))ds$. Define $\nu:=\inf\{t\geq0: Y(\omega,z(s)) = 0\}$. Then, $z'(t)>0$ for any $t<\nu$. Thus, $z$ is strictly increasing on $[0,\nu)$. The inverse function $f$ of $z\vert_{[0,\nu)}$ is absolutely continuous and the transformation formula yields $f(t) = \int_0^t \frac{1}{Y(\omega,s)}ds = I(\omega,t)$ and, hence, $z(t)=g(\omega,t)=\alpha_1(\omega,t)=\alpha_2(\omega,t)$ for $t< \nu$. Since $z$ is non-decreasing we have $\alpha_1(\omega,t)\leq z(t)$ for any $t\geq 0$.

We have
\begin{align*}
  t &\geq \lim_{\epsilon\searrow0} \int_0^t \frac{Y(\omega,z(s))}{Y(\omega,z(s))\vee \epsilon}ds \\
    &= \lim_{\epsilon\searrow0} \int_0^{z(t)}\frac{1}{Y(\omega,u)\vee \epsilon}du \\
    &= \int_0^{z(t)}\frac{1}{Y(\omega,u)}du \\
    &= \lim_{\delta\searrow0} I(\omega,z(t)-\delta)
\end{align*}
  for any $t\geq 0$ with $z(t)>0$. Moreover, we have $t=I(\omega,\alpha_2(\omega,t))$ for $t<\gamma(\omega)$ and, hence, $z(t) \leq \alpha_2(\omega,t)$ because $I(\omega,\cdot)$ is a strictly increasing function.
  
The two final statements of the claim follow directly from the construction of $\alpha_1,\alpha_2$ or are trivial.
\end{proof}

Next, we relate TCEs with IVPs.
\begin{prop}\label{p:TCE and IVP}
  Let $g:E\rightarrow\mathbb R_+$ be measurable.
  \begin{itemize}
    \item If $\tau:\Omega\times\mathbb R_+\rightarrow\mathbb R_+$ is a measurable function such that $\tau(t) = \int_0^t g(X(\tau(s)))ds$ for any $t\geq0$, then $Z(t):=X(\tau(t))$ is a solution of the TCE $(X,g)$.
    \item If $Z$ is a solution of the TCE $(X,g)$ and $\tau(t):= \int_0^t g(Z(s)) ds$, then
      $\tau(t) = \int_0^t g(X(\tau(s)))ds$ for any $t\geq0$ $P$-a.s.
  \end{itemize}
\end{prop}
\begin{proof}
 Let $\tau(t) = \int_0^t (g\circ X)(\tau(s))ds$ for any $t\geq0$ $P$-a.s.\ and define $Z(t):=X(\tau(t))$. Then, we have
  $$ Z(t) = X(\tau(t)) = X\left( \int_0^t g(Z(s)) ds \right) $$
 for any $t\geq0$ up to a $P$-null set.
 
 Now, let $Z$ be a solution of the TCE $(X,g)$ and $\tau(t):= \int_0^t g(Z(s)) ds$. Then, we have
 \begin{align*}
    Z(t) &= X\left( \int_0^tg(Z(s)) ds \right) = X(\tau(t)), \\
    \tau(t) &= \int_0^t g(Z(s)) ds = \int_0^t g(X(\tau(s))) ds
 \end{align*}
for any $t\geq 0$.
\end{proof}
Due to the previous Proposition, the solution theory for TCEs can be reduced to the solution theory of homogeneous ODEs on $\mathbb R_+$ 
or, to be precise, integral equations -- which is much simpler, of course.
\deleted{evtl sollten wir das mit den integral equations mal sagen. Fuer klassische Analytiker-Puristen machen Differential- und Integralgleichung ja einen Unterschied...}

\begin{cor}\label{k:TCE/IVP II}
  Let $g:E\rightarrow\mathbb R_+$ be measurable, bounded and regular at zero, i.e.\ for any $x\in E$ with $g(x)>0$ there are $\epsilon,\delta>0$ such that $g(y)>\epsilon$ for any $y\in B(x,\delta)$. \deleted{Define $Y(\omega,t):=g(X(\omega,t))$ for any $t\geq 0$, $\omega\in\Omega$.} Then, the TCE $(X,g)$ has a solution.
  
  Let $x\in E$ and \added{define $Y(\omega,t):=g(X(\omega,t))$ for any $t\geq 0$, $\omega\in\Omega$}. The TCE $(X,g)$ has a unique solution up to $P_x$-indistinguishability if and only if there is a $P_x$-null set $N\subseteq \Omega$ such that the IVP $(Y(\omega,\cdot),0)$ has a unique solution for any $\omega\in \Omega\backslash N$.
\end{cor}
\begin{proof}
  {\em Existence:} Observe, that $Y$ is right regular at zero. Thus Proposition \ref{p:Existenz IVP} yields a measurable function $\tau:\Omega\times\mathbb R_+\rightarrow\mathbb R_+$ such that $\tau(t) = \int_0^t Y(\tau(s)) ds$. Proposition \ref{p:TCE and IVP} yields a solution to the TCE $(X,g)$.
  
  {\em Uniqueness:} Let $x\in E$.
  
  We start by assuming that the TCE $(X,g)$ has a unique solution $Z$ up to $P_x$-indistinguishability. Let $\alpha_1,\alpha_2$ be the functions given in Proposition \ref{p:Existenz IVP}. Define $Z_i(t):=Y(\alpha_i(t))$. Then, Proposition \ref{p:TCE and IVP} yields that $Z_1$ and $Z_2$ are solutions to the TCE $(X,g)$. Consequently, $Z_1=Z=Z_2$ up to a $P_x$-null set $N\subseteq \Omega$. Proposition \ref{p:TCE and IVP} states that $\alpha_1(t) = \int_0^ tg(Z(s))ds = \alpha_2(t)$ for any $t\geq 0$.
  
  Now assume that there is a $P_x$-null set $N\subseteq \Omega$ such that the IVP $(Y(\omega,\cdot),0)$ has a unique solution $\tau(\omega,\cdot)$ for any $\omega\in \Omega\backslash N$. Define $\tau(\omega,t):=0$ for any $t\geq 0$, $\omega\in N$ and let $\alpha_1$ be the measurable solution given in Proposition \ref{p:Existenz IVP}. Then $\tau=\alpha_1$ on $\Omega\backslash N$ by assumption. Let $Z$ be any solution to the TCE $(X,g)$. Then, Proposition \ref{p:TCE and IVP} yields that $\beta(t):=\int_0^tg(Z(s))ds$ is a solution to the IVP $(Y,0)$. Hence, $\beta=\tau=\alpha_1$ on $\Omega\backslash N$. Thus, Proposition \ref{p:TCE and IVP} yields $Z(t) = X(\alpha_1(t))$ for any $t\geq 0$ on $\Omega\backslash N$.
\end{proof}

With the connection of TCEs and IVPs at hand we can now state a combined result of the previous three statements. This is a variation of \cite[Theorem 1.1, Ch.6]{ethier.kurtz.86} which has a more relaxed condition on the stopping times but is more demanding on the given process $X$, namely 
\begin{thm}\label{t:TCE eindeutig loesbar}
  Let $g:E\rightarrow\mathbb R_+$ be measurable, bounded and regular at zero and let $\tau_0:=\inf\{t\geq0: g(X(s))=0\}$. Assume that we have
    $$ \int_{\tau_0}^{\tau_0+\epsilon}\frac{1}{g(X(s))}ds=\infty$$
    $P_x$-a.s.\ on $\{\tau_0<\infty\}$ for any $x\in E$.
 
 Then the TCE $(X,g)$ has a solution $Z$ such that this solution is unique up to $P_x$-indistinguishability for any $x\in E$.
\end{thm}
\begin{proof}
  Corollary \ref{k:TCE/IVP II} yields a solution $Z$ for the TCE $(X,g)$, i.e.\  $$ Z(\omega,t) = X\left( \int_0^tg(Z(\omega,s))ds \right) $$
  for any $\omega\in\Omega$, $t\geq 0$. Define $Y(\omega,t):=g(X(\omega,t))$. Then, $Y$ is right regular at zero, measurable and bounded.
  
 Let $\alpha_1,\alpha_2$ be the functions given in Proposition \ref{p:Existenz IVP}. Then, Proposition \ref{p:Existenz IVP} states that $\alpha_1$ is the unique solution to the IVP $(Y,0)$ up to a $P_x$-null set $N_x\subseteq\Omega$ for any $x\in E$. Corollary \ref{k:TCE/IVP II} states that $Z$ is the unique solution to the TCE $(X,g)$ up to $P_x$-indistinguishability.
\end{proof}

If one has knowledge of the occupation measure for the original process $X$, then it can be used to verify the requirements of Theorem \ref{t:TCE eindeutig loesbar}.
\begin{cor}
  Let $g:E\rightarrow\mathbb R_+$ be measurable, bounded and regular at zero and let
   $$ \Lambda_t(A) := \int_0^t 1_{\{X(s)\in A\}} ds,\quad t\geq 0, A\subseteq E\text{ mb.}$$
   be the occupation measure of $X$. Assume that
    $$ \int_E \frac{1}{g(y)} \Lambda_t(dy) = \infty,\quad P_x\text{-a.s.} $$
   for any $t>0$, $x\in g^{-1}(\{0\})$. Then the requirements of Theorem \ref{t:TCE eindeutig loesbar} are met.
\end{cor}
\begin{proof}
  By the Markov property of $X$ it suffices to show that
   $$ \int_0^t \frac{1}{g(X(s))} ds = \infty,\quad P_x-\text{a.s.} $$
  for any $t>0$ and any $x\in g^{-1}(\{0\})$. Let $t>0$ and $x\in g^{-1}(\{0\})$. Then, we have
  \begin{align*}
    \int_0^t \frac{1}{g(X(s))} ds &= \int_E \frac{1}{g(y)} \Lambda_t(dy) \\
                                  &= \infty ,\quad P_x\text{-a.s.}
  \end{align*}
  by assumption and hence the requirements of Theorem \ref{t:TCE eindeutig loesbar} are met.
\end{proof}

The following corollary links uniqueness of the TCE $(X,g)$ with path properties of the underlying process $X$ and growth properties of the function $g$ from its zeros.
\begin{cor}\label{k:hinreichende Bedingungen}
 Let $g:E\rightarrow\mathbb R_+$ be measurable, bounded and regular at zero in the sense of Corollary \ref{k:TCE/IVP II}, $\lambda>0$ and define $\tau_0:=\inf\{t\geq0: g(X(s))=0\}$. Assume the following two statements
  \begin{itemize}
   \item[(A1)] For any $x\in E$ with $g(x)=0$ there are constants $C_g>0,\delta>0$ such that for any $y\in B(x,\delta)\cap E$ we have
    $$ \vert g(y)\vert \leq C_g\vert y-x\vert^\lambda. $$
   \item[(A2)] For any $x\in E$ there are random variables $C,\delta:\Omega\rightarrow\mathbb R_+$ such that for any $\omega\in\{\tau_0<\infty\}$ and any $t\in[\tau_0(\omega),\tau_0(\omega)+\delta(\omega)]$ we have
     $$ \vert \replaced{X(\omega,\tau_0(\omega))}{X(\tau_0(\omega),\omega)} - \replaced{X(\omega,\tau_0(\omega)+t)}{X(\tau_0(\omega)+t,\omega)} \vert \leq C(\omega) t^{1/\lambda}$$
     $P_x$-a.s.
  \end{itemize}
  Then, the TCE $(X,g)$ has a unique solution.
\end{cor}
\begin{proof}
 We show that the requirements of Theorem \ref{t:TCE eindeutig loesbar} are met. Let $x\in E$ and $N_x$ be the $P_x$-null set outside which (A2) holds. Let $\omega\in\{\tau_0<\infty\}\backslash N_x$ and $\epsilon>0$. We have
 \begin{align*}
    \vert X(\omega,\tau_0(\omega)+s) - X(\omega,\tau_0(\omega)) \vert &\leq C(\omega)s^{1/\lambda}
 \end{align*}
 for any $s\in (0,\delta(\omega))$ by (A2). Thus, we have
 \begin{align*}
   \vert g(X(\omega,\tau_0(\omega)+s)) \vert & \leq C_g C^\lambda(\omega) s
 \end{align*}
  for any $s>0$ with $s < (\delta/C(\omega))^\lambda$ where $\delta,C_g$ are chosen relative to the point $x:=X(\omega,\tau_0(\omega))$ as in (A1). Thus, we have
 \begin{align*}
   \int_{\tau_0(\omega)}^{\tau_0(\omega)+\epsilon} \frac{1}{g(X(\omega,s))} ds &= \int_{0}^{\epsilon} \frac{1}{g(X(\omega,\tau_0(\omega)+s))} ds \\
   &\geq \int_0^{\eta(\omega)} \frac{1}{C_g C^\lambda(\omega) s} ds \\
   &= \infty
 \end{align*}
  where $\eta(\omega):=\epsilon\wedge (\delta/C(\omega))^\lambda$.
\end{proof}

\begin{prop}\label{p:regular solutions}
  Let $g:E\rightarrow\mathbb R_+$ be bounded and measurable and assume that the TCE $(X,g)$ has a unique solution $Z$. Then, $Z$ is a process with symbol $\tilde q(x,u):=g(x)q(x,u)$ for any $x,\in E$, $u\in\mathbb R^d$.
  
If additionally the law of $X$ is determined by $q$ and $\tilde q$ is continuous, then $Z$ is a strong Markov process relative to $(\Omega,\mathcal A,(\mathcal F^Z_t)_{t\geq 0},(P_x)_{x\in E})$ where $(\mathcal F^Z_t)_{t\geq 0}$ denotes the right continuous filtration generated by $Z$.
\end{prop}
\begin{proof}
 Let $x\in E$, $u\in\mathbb R^d$. Define $\alpha(t):=\int_0^t g(Z(s))ds$ for any $t\geq 0$. Then, $\alpha$ is absolutely continuous and $\alpha'(t) = g(Z(t))$. We have
  \begin{align*}
    \E_x(e^{i\<u,Z(t)\>}\vert \mathcal F^Z_s) &= \E_x\left( e^{i\<u,X(\alpha(t))\> } \vert \mathcal F^Z_s \right) \\
                           &= \E_x\left( \int_0^{\alpha(t)} q(X(s),u)e^{i\<u,X(s)\>}ds \vert \mathcal F^Z_s\right) \\
                           &= \E_x\left( \int_0^tq(X(\alpha(r)),u)e^{i\<u,X(\alpha(r))\>}\alpha'(r)dr \vert \mathcal F^Z_s\right) \\
                           &= \E_x\left( \int_0^t\tilde q(Z(r),u)e^{i\<u,Z(r)\>}dr \vert \mathcal F^Z_s\right)
  \end{align*}  
  for any $t\geq 0$ where we used optional stopping for the second equation and the substitution formula for the third equation. In the same way one can show that
   $$ \E_x\left(\int_0^t \vert\tilde q(Z(s),u)\vert ds \right) = \E_x\left(\int_0^{\alpha(t)} \vert q(X(s),u)\vert ds \right) < \infty $$
   for any $t\geq 0$ because $\alpha(t) \leq t\Vert g\Vert_\infty$.
  
Now let us assume additionally that the law of $X$ is determined by $q$ and $\tilde q$ is continuous. Define the mappings
 \begin{align*}
   f_u&:E\rightarrow\mathbb C, x\mapsto e^{i\<u,x\>}, \\
   A &:= \{ (f_u,q(\cdot,u)f_u) : u\in \mathbb R^d \}, \\
   A_g &:= \{ (f_u,\tilde q(\cdot,u)f_u) : u\in \mathbb R^d \}.
 \end{align*}
 Then, $A_g\subseteq \overline C(E)\times \overline C(E)$ and $A\subseteq \overline C(E)\times B(E)$ by assumptions on the symbol $q$ and $Z$ is a solution to the martingale problem $(A_g,\delta_x)$ under $P_x$ by \cite[Theorem 1.3, Ch.6]{ethier.kurtz.86}. Let $Z_2$ be another solution of the martingale problem $(A_g,\delta_x)$. Then, \cite[Theorem 1.4, Ch.6]{ethier.kurtz.86} yields that there is a version of $Z_2$, also denoted by $Z_2$, which is a solution to the TCE $(Y,g)$ where $Y$ is an other solution to the martingale problem $(A,\delta_x)$. However, $Z_2=F(Y)$ for a measurable function $F:\mathbb D([0,\infty),E)\rightarrow \mathbb D([0,\infty),E)$ and, hence, we have
  $$ P^{(Y,Z_2)} = P^{(Y,F(Y))} = P^{(X,F(X))} = P^{(X,Z)}. $$
 Hence, the law of $Z_2$ and $Z$ coincide, i.e.\ the martingale problem for $(A_g,\delta_x)$ is well posed.\cite[Theorem 4.2, Ch.4]{ethier.kurtz.86} yields that the martingale problem for $A_g$ is well posed and together with \cite[Theorem 4.6, Ch.4]{ethier.kurtz.86} that $Z$ is a strong Markov process.
\end{proof}

For a stochastic process $Y$ with values in $\mathbb R^d$ we will also use the notation
 $$ \vert Y\vert^*_t:=\sup_{s\in[0,t]}\abs{Y(s)} $$
 for any $t\geq 0$ where $\abs{\cdot}$ denotes the Euclidean norm.

\begin{prop}\label{p:(A2) holds}
Let $\beta_\infty$ be the uniform index for $X$ and $\tau$ be a stopping time. Then, we have
  \begin{itemize}
    \item[(A2*)] For any $x\in E$, $\epsilon>0$ there is a random variable $\delta:\Omega\rightarrow\mathbb R_+$ such that for any $\omega\in\{\tau<\infty\}$ and any $t\in[\tau(\omega),\tau(\omega)+\delta(\omega)]$ we have
     $$ \vert X(\omega,\tau(\omega)) - X(\omega,\tau(\omega)+t) \vert \leq \epsilon t^{1/\lambda}$$
     $P_x$-a.s.
  \end{itemize}
\end{prop}
\begin{proof}
We start to show the claim for stopping times which are $P_x$-a.s.\ finite, i.e.\ we assume in step 1 to 3 that $P (\tau=\infty)=0$.

{\em Step 1:}
At first we show for $h,R>0$ and any probability measure $\mu$ on $\bbr^d$
\[
  P_\mu(\vert X-X(0)\vert_h^*\geq R) \leq c_d h H(R).
\]
where $c_d$ depends only on the dimension $d$. By \cite[Proposition 3.10]{schnurr.13} we know that the inequality holds under $P_x$ for any $x\in\bbr^d$. Recall that $H(R)=\sup_{x\in\bbr} H(x,R)$. The desired inequality follows by 
\[
  P_\mu(\vert X-X(0)\vert_h^*\geq R)=\int_{\bbr^d} P_x(\vert X-x\vert_h^*\geq R) \mu(dx) \leq c_d h H(R)
\]
where we have used $P_x(X(0)=x)=1$.\\
{\em Step 2:} Next we show for every $h,R>0$
\begin{align} \label{eq:transfer}
  P_\mu(\vert X-X(0)\vert_h^*\geq R) = P_x(\vert X(\tau+(\cdot))-X(\tau)\vert^*\geq R)
\end{align} 
where we denote by $\mu:=P_x^{X(\tau)}$ the law of $X(\tau)$ under $P_x$. Clearly, we have
\begin{align} 
  \{ (\vert X(\tau+(\cdot))-X(\tau)\vert_h^* \geq R \} &=  \left\{ \sup_{0\leq s \leq h, s \in\bbq} \abs{X(\tau+s)-X(\tau)} \geq R \right\}  \nonumber \\ 
	&= \bigcap_{n\in\mathbb N}\bigcup_{0\leq s \leq h, s\in\bbq} \left\{\abs{ X(\tau+s) - X(\tau) } \geq R-1/n \right\}. \label{eq:repres}
\end{align}
Let us write $Q^n_{\tau,s}:= \{ \abs{X(\tau+s)-X(\tau) } \geq R-1/n \}$ as well as $Q^n_{0,s}:=\{ \abs{X(s)-X(0) } \geq R-1/n \}$ and let $(s_j)_{j\in\bbn}$ be an enumeration of $(0,h]\cap \bbq$ and $s_0:=0$. By the strong Markov property \cite[Equation (1.17), Ch.2]{ethier.kurtz.86} we obtain 
$$P_x^{(X(\tau+s_0),\dots,X(\tau+s_m))} = P_\mu^{(X(s_0),\dots,X(s_m))}$$
for any $m\in\bbn$ and, hence,
\[
  P_x(Q^n_{\tau,s_1} \cup \dots \cup Q^n_{\tau,s_m}) = P_\mu(Q^n_{0,s_1} \cup \dots \cup Q^n_{0,s_m}).
\]
Letting $m\to\infty$ and using the continuity from below of $P_\mu$ and $P_x$ we get 
\[
  P_x \left(\bigcup_{0\leq s \leq h, s\in\bbq} Q^n_{\tau,s}\right) = P_\mu\left( \bigcup_{0\leq s \leq h, s\in\bbq} Q^n_{0,s} \right)
\]
and finally the result by taking a countable intersection as in the representation \eqref{eq:repres} and using continuity from above of the two measures under consideration. \\
{\em Step 3:} Putting the above together, we obtain
\[
  P_x(\vert X(\tau+(\cdot))-X(\tau)\vert_h^*\geq R) \leq c_d h H(R).
\]
Using exactly the same Borel-Cantelli argument as in the proof of \cite[Theorem 3.12]{schnurr.13}, replacing $H(x,R)$ by $H(R)$, we obtain
\[
  \lim_{h\to 0} h^{-1/\lambda} \vert X(\tau+(\cdot))-X(\tau)\vert_h^* = 0
\]
$P_x$-a.s. Hence, there is a $P_x$-null set $N_x$ such that for every $\varepsilon>0$ there exists an $(0,\infty)$-valued random variable $\delta$ such that for any $h\in [\tau(\omega),\tau(\omega)+\delta(\omega)[$ we have:
\[
  h^{-1/\lambda}\abs{(X(\omega,\tau(\omega)+h)-X(\omega,\tau(\omega))} \leq h^{-1/\lambda}\vert X(\omega,\tau(\omega)+(\cdot)) - X(\tau)\vert_h^* \leq \varepsilon 
\]
for $\omega\in\Omega\backslash N_x$, respectively
\[
  \abs{X(\omega,\tau(\omega)+h)-X(\omega,\tau(\omega))} \leq \varepsilon h^{1/\lambda}.
\]

{\em Step 4:} Now assume that $\tau$ is any stopping time and define $\tau_k:=\tau\wedge k$ for $k\in\mathbb N$ which is clearly finite valued. Let $N_k$ be the $P_x$-null set such that the claim holds for $\tau_k$ outside $N_k$ and define the $P_x$-null set $N:=\bigcup_{k\in\mathbb N}N_k$. Let $\epsilon >0$ and $\delta_k$ be the random variable satisfying the claim in accordance with $\tau_k$ and define the random variable
 $$ \delta := \delta_11_{\{\tau\leq 1\}} + \sum_{k\in\mathbb N} \delta_k1_{\{\tau\in(k,k+1]\}}.$$

Let $\omega\in \{\tau<\infty\}\backslash N$. Then, there is a minimal $k\in\mathbb N$ such that $\tau(\omega)<k$ and, hence, $\tau_k(\omega)=\tau(\omega)$. However, we have
 $$ \abs{X(\omega,\tau(\omega)+h)-X(\omega,\tau(\omega))} = \abs{X(\omega,\tau_k(\omega)+h)-X(\omega,\tau_k(\omega))} \leq \varepsilon h^{1/\lambda}$$
 for $h\in [\tau_k(\omega),\tau_k(\omega)+\delta_k(\omega)]=[\tau(\omega),\tau(\omega)+\delta(\omega)]$.
\end{proof}


\begin{thebibliography}{10}

\bibitem{belomestny}
D. Belomenstny, Statistical inference for time-changed L\'evy processes via composite characteristic function estimation, 
\emph{Ann. Statist.} \textbf{39(4)} (2011), 2205--2242.

\bibitem{blumenthalget}
Blumenthal, R.~M. and Getoor, R.~K.
\newblock (1968). {\em Markov {P}rocesses and {P}otential {T}heory}.
\newblock New York, Academic Press.

\bibitem{boettcher.al.13}
B.~B\"ottcher, R.L.~Schilling and J.~Wang, \emph{{L}\'evy matters {III}},
  Springer, Switzerland, 2013.

\bibitem{dengschilling}
C.-S. Deng and R.L. Schilling, On shift Harnack inequalities for subordinate semigroups and moment estimation for L\'evy processes.
 \emph{Stoch. Proc. Appl.} \textbf{125(10)} (2015), 3851--3878.

\bibitem{doering.15}
L.~D\"oring, \emph{A jump-type sde approach to real-valued self-similar Markov
  processes}, To appear, 2015.

\bibitem{dorroh.66}
J.~R. Dorroh, \emph{Contraction semigroups in a function space}, Pacific
  Journal of Mathematics \textbf{19} (1966), 35--38.

\bibitem{engelbert.schmidt.85}
H.~J. Engelbert and W.~Schmidt, \emph{On solutions of stochastic differential
  equations without drift}, Zeitschrift f\"ur Wahrscheinlichkeitstheorie und
  verwandte Gebiete \textbf{68} (1985), 287--317.

\bibitem{ethier.kurtz.86}
S.~Ethier and T.~Kurtz, \emph{Markov {P}rocesses. {C}haracterization and
  {C}onvergence}, Wiley, New York, 1986.

\bibitem{fuglede}
Fuglede, B. 
\newblock (1972). {\em Finely {H}armonic {F}unctions}.
\newblock Berlin, Springer.

\bibitem{gabrielli.teichmann.15}
A.~Gabrielli and J.~Teichmann, \emph{Pathwise construction of affine
  processes}, 2015.

\bibitem{gustafson.lumer.72}
K.~Gustafson and G.~Lumer, \emph{Multiplicative pertubation of semigroup
  generators}, Pacific journal of mathematics \textbf{41} (1972), no.~3.

\bibitem{lamperti.72}
J.~Lamperti., \emph{Semi-stable markov processes. I.}, Zeitschrift f\"ur Wahrscheinlichkeitstheorie und Verwandte Gebiete (1972), no.~22, 205--225.

\bibitem{jacob.93}
N.~Jacob, \emph{Further pseudodifferential operators generating Feller semigroups and Dirichlet forms}, Revista matematica Iberoamericana
  \textbf{9} (1993), no.~2.

\bibitem{jacob.schilling.01}
N.~Jacob and R.L. Schilling, \emph{{L{\'e}vy-type processes and
  pseudo-differential operators}}, L{\'e}vy processes: theory and applications
  (2001), 139--167.

\bibitem{js.87}
J.~Jacod and A.~Shiryaev, \emph{Limit {T}heorems for {S}tochastic {P}rocesses},
  second ed., Springer, Berlin, 2003.

\bibitem{kallsen.04}
J.~Kallsen, \emph{A didactic note on affine stochastic volatility models}, From
  Stochastic Calculus to Mathematical Finance (Yu. Kabanov, R.~Liptser, and
  J.~Stoyanov, eds.), Springer, Berlin, 2006, pp.~343--368.

\bibitem{karatzas.shreve.98}
I.~Karatzas and S.~Shreve, \emph{Methods of mathematical finance}, Springer,
  Berlin, 1998.

\bibitem{schilling98}
R.L.~Schilling, Growth and H\"older conditions for the sample paths of Feller processes.
\emph{Prob. Theory. rel. Fields} \textbf{112} (1998), 565--611. 

\bibitem{schilling.et.al.12}
R.L.~Schilling, R.~Song and Z.~Vondra\v{c}ek, \emph{Bernstein Functions: Theory and Applications}, 2nd edn. De Gruyter, Berlin, 2012

\bibitem{schnurr.13}
A.~Schnurr, \emph{Generalization of the {B}lumenthal-{G}etoor index to the
  class of homogeneous diffusions with jumps and some applications}, Bernoulli
  \textbf{19} (2013), no.~5A.

\bibitem{volkonskii.58}
V.~A. Volkonskii, \emph{Random substitution of time in strong {Markov}
  processes}, Theory of Probability and its Applications \textbf{3} (1958),
  no.~3, 310--326.

\end{thebibliography}

\providecommand{\bysame}{\leavevmode\hbox to3em{\hrulefill}\thinspace}
\providecommand{\MR}{\relax\ifhmode\unskip\space\fi MR }
\providecommand{\MRhref}[2]{%
  \href{http://www.ams.org/mathscinet-getitem?mr=#1}{#2}
}
\providecommand{\href}[2]{#2}

\end{document}